\documentclass[11pt]{elsarticle}

\usepackage{amsmath,amssymb,amsthm} 
\usepackage{mathrsfs}
\usepackage{xspace}
\usepackage{enumitem}
\usepackage{longtable}
\usepackage{array,multirow}
\newcolumntype{L}{>{$}l<{$}}
\newcolumntype{C}{>{$}c<{$}}
\newcolumntype{R}{>{$}r<{$}}

\newcommand{\url}[1]{\texttt{#1}}

\newtheorem{thm}{Theorem}

\newtheorem{prop}[thm]{Proposition}

\theoremstyle{definition}

\newcommand{\ie}{\textit{i.e.}\xspace}

\newcommand{\Z}{\ensuremath{\mathbb{Z}}\xspace}
\newcommand{\R}{\ensuremath{\mathbb{R}}\xspace}

\newcommand{\Sig}{\mathfrak{S}}

\renewcommand{\leq}{\leqslant}
\renewcommand{\geq}{\geqslant}

\renewcommand{\epsilon}{\varepsilon}
\renewcommand{\phi}{\varphi}
\renewcommand{\rho}{\varrho}

\newcommand{\Restriction}{\mathpunct{\mathpunct{\restriction}}}
\newcommand{\suchthat}{\ |\ }

\newcommand{\CORRECTION}[1]{#1}
\newcommand{\CORRECT}[1]{#1}

\newcommand{\vh}{\beta}
\newcommand{\varphiVH}{\vh}
\usepackage{bm}
\newcommand{\gene}[1]{\ensuremath{\bm{\mathsf{#1}}}}
\usepackage{tikz-cd}

\usepackage{listings}
\usepackage{xcolor}
\definecolor{codegreen}{rgb}{0,0.6,0}
\definecolor{codegray}{rgb}{0.7,0.7,0.7}
\lstdefinestyle{mystyle}{
    commentstyle=\color{codegreen},
    keywordstyle=\color{red},
    numberstyle=\tiny\color{codegray},
    basicstyle=\footnotesize\ttfamily,
    breakatwhitespace=false,
    breaklines=true,
    captionpos=b,
    keepspaces=true,
    numbers=none, 
    numbersep=5pt,
    showspaces=false,
    showstringspaces=false,
    showtabs=false,
    tabsize=2
}
\begin{document}
\begin{frontmatter}
\title{From multivalued to Boolean functions: preservation of soft nested canalization}
\author[1]{Élisabeth Remy}
\ead{elisabeth.remy@univ-amu.fr}
\ead[url]{https://mabios.math.cnrs.fr/}
\author[2]{Paul Ruet\corref{cor1}}
\ead{ruet@irif.fr}
\ead[url]{https://www.irif.fr/~ruet/}
\affiliation[1]{organization={Aix Marseille Univ, CNRS, I2M}, city={Marseille}, country={France}}
\affiliation[2]{organization={CNRS, Université Paris Cité}, city={Paris}, country={France}}
\cortext[cor1]{Corresponding author}
\begin{abstract}
Nested canalization (NC) is a property of Boolean functions which has been recently extended to multivalued functions. We study the effect of the Van Ham mapping (from multivalued to Boolean functions) on this property. We introduce the class of softly nested canalizing (SNC) multivalued functions, and prove that the Van Ham mapping sends SNC multivalued functions to NC Boolean functions. Since NC multivalued functions are SNC, this preservation property holds for NC multivalued functions as well. We also study the relevance of SNC functions in the context of gene regulatory network modelling.
\end{abstract}
\end{frontmatter}

\section{Introduction}

Dynamical properties of biological systems, such as stability and robustness, are often associated with the canalizing property. This notion is at the basis of the work of Waddington in embryology, which describes, in the 1940s, epigenetic landscapes representing embryogenesis by canalizing configurations \cite{Wad42}. In the 1990’s, S. Kauffman introduced the class of canalizing Boolean functions \cite{Kau93,Kau03} to formalize the canalizing behaviour observed in gene regulatory networks. In short, canalizing Boolean functions are functions $f:(\Z/{2\Z})^n\rightarrow\Z/{2\Z}$ (or $\R$) such that at least one input variable, say $x_i$ ($1\leq i\leq n$), has a value $a=0$ or $1$ which determines the value of $f(x)$. Nested canalizing (NC) functions provide a ``recursive'' version of canalizing functions: an NC function $f$ is canalizing and, moreover, its restriction $f\Restriction_{x_i\neq a}$ is itself NC.
NC functions are particularly interesting because of their ``low complexity'': in particular, their average sensitivity (a measure of complexity related to spectral concentration, learning properties, decision tree complexity \cite{ODon14} and stability properties \cite{Kau93,Kau04}) is bounded above by a constant \cite{Lau13,Klo13}.

A natural question is the extension of this canalizing property to multivalued functions, involving more than two expression levels, which are often needed in the modeling of biological systems to circumvent the too restrictive Boolean representation \cite{Tho91}. The present paper investigates the property of nested canalization for multivalued functions $f:(\Z/{k\Z})^n\rightarrow\Z/{k\Z}$ or $\R$ for some $k\geq 2$. A definition of nested canalizing (NC) multivalued function has been proposed in \cite{ML11,ML12}. In \cite{RR22}, we introduced a more general class of multivalued functions, which we called weakly nested canalizing (WNC), for which we proved that the average sensitivity is bounded above by a constant.

Multivalued modelling allows to distinguish different thresholds of action of a gene on its targets, and is of course more complex to analyse than the Boolean case. Boolean mappings have been proposed, associating one Boolean variable to each threshold. \CORRECTION{In this way, the various actions of the multivalued component in the network are deployed, each represented by a Boolean variable. The behavior of the Boolean system thus obtained can be related to that of the initial multivalued model}. This gives access to a large number of theoretical works and numerical tools specific to the analysis of Boolean models. Thus we are interested in the relationship between the notions of canalization for multivalued and Boolean functions. Clearly an NC Boolean function is NC multivalued, but what can we say about the other direction? More precisely, a mapping $\vh$ from multivalued functions to Boolean functions has been defined by Van Ham in \cite{vH79}, and it has been proved in \cite{DRC11,Ton19} to be in some sense unique (it is the only injective such mapping preserving neighbours in the state space, and for which the regulatory graphs are unchanged).

We prove here that $\vh$ maps NC multivalued functions to NC Boolean functions. To this end, we introduce in Section \ref{secncf} the notion of softly nested canalizing (SNC) multivalued functions, and prove that NC $\Rightarrow$ SNC $\Rightarrow$ WNC (Proposition \ref{prop:ncsnc} and Section \ref{sec:wnc}). Then we actually prove in Theorem \ref{th:nc} that SNC functions are mapped by $\vh$ to NC Boolean functions. On the other hand, we show that Theorem \ref{th:nc} does not extend to WNC functions: we shall see in Section \ref{contrex} an example of a WNC function which is mapped by $\vh$ to a non NC Boolean function.

NC functions appear as appropriate rules in Boolean models of gene regulatory networks. Indeed, Boolean networks have been proposed as models for gene regulatory networks \cite{Tho73}, with different nodes corresponding to different genes. The activity of a gene is regulated by the activity of other genes to which it is connected through a Boolean function. These networks are built from biological data and knowledge from the literature, and their associated Boolean functions are far from random. It turns out that NC functions are predominant in the large databases of Boolean gene networks \cite{Sub22}. In Section \ref{examples} we extract the multivalued genes from this database, and we study the canalizing properties of their multivalued functions and their corresponding Boolean functions.

\section{Nested canalizing multivalued functions}
\label{secncf}

The concept of nested canalization, originally considered for Boolean functions \cite{Kau93}, has been extended to multivalued functions in \cite{ML11,ML12,KLKAL17}. We first recall here this definition, before introducing soft nested canalization in Section \ref{secsoft}.

Let $k,n$ be positive integers, $k\geq 2$. $\Z/{k\Z}$ is the ring of integers modulo $k$.
A function $f:(\Z/{k\Z})^n\rightarrow\Z/{k\Z}$ is said to be \emph{canalizing with respect to coordinate $i$ and $(a,b)\in\Z/{k\Z}\times\Z/{k\Z}$} if there exists a function $g:(\Z/{k\Z})^n\rightarrow\Z/{k\Z}$ different from the constant $b$ such that
$$
f(x)=
\begin{cases}
b & \text{if } x_i=a \\
g(x) & \text{if } x_i\neq a.
\end{cases}
$$
We shall simply say that $f$ is \emph{canalizing} if it is canalizing with respect to some $i,a,b$.

A \emph{segment} is a \CORRECTION{(nonempty, strict)} subset of $\Z/{k\Z}$ of the form $\{0,\ldots,i\}$ or $\{i,\ldots,k-1\}$, with $0\leq i\leq k-1$. Let $\sigma\in\Sig_n$ be a permutation, $A_1,\ldots,A_n$ be segments, and $c_1,\ldots,c_{n+1}\in\Z/{k\Z}$ be such that $c_n\neq c_{n+1}$. Then $f$ is said to be \emph{nested canalizing (NC) with respect to $\sigma$, $A_1,\ldots,A_n$, $c_1,\ldots,c_{n+1}$} if 
$$
f(x)=
\begin{cases}
c_1 & \text{if } x_{\sigma(1)}\in A_1 \\
c_2 & \text{if } x_{\sigma(1)}\notin A_1,x_{\sigma(2)}\in A_2 \\
\;\vdots & \quad\vdots \\
c_n & \text{if } x_{\sigma(1)}\notin A_1,\ldots,x_{\sigma(n-1)}\notin A_{n-1},x_{\sigma(n)}\in A_n \\
c_{n+1} & \text{if } x_{\sigma(1)}\notin A_1,\ldots,x_{\sigma(n-1)}\notin A_{n-1},x_{\sigma(n)}\notin A_n.
\end{cases}
$$
We shall simply say that $f$ is \emph{NC} if it is NC with respect to some $\sigma$, $A_1,\ldots,A_n$, $c_1,\ldots,c_{n+1}$.

\CORRECTION{In Sections \ref{sec:reduc} and \ref{examples}, we shall be interested in multivalued networks. These networks are represented by functions $F:(\Z/{k\Z})^n\rightarrow(\Z/{k\Z})^n$, where for each state $x\in(\Z/{k\Z})^n$ and each $i\in\{1,\ldots,n\}$, $F_i(x)\in\Z/{k\Z}$ gives the target value of gene number $i$. We shall say that such a network $F$ is \emph{NC} if all functions $F_i:(\Z/{k\Z})^n\rightarrow\Z/{k\Z}$ are.}

Note that the segment $A_i$ plays the role of a threshold for the variable $x_{\sigma(i)}$. For an NC multivalued function $f$ as defined in \cite{ML11,ML12,KLKAL17}, each variable $x_{\sigma(i)}$ is used exactly once in the process of canalization (and the value of $f$ is determined by the threshold corresponding to $A_i$). It is possible to relax the definition by allowing canalization along successive thresholds for a same variable. This is essentially what we shall do in defining softly nested canalizing functions below. In doing so, we shall also remove the condition $c_n\neq c_{n+1}$.

\subsection{Softly nested canalizing multivalued functions}
\label{secsoft}

Let $n$ be a positive integer. For each $i\in\{1,\ldots,n\}$, let $k_i>0$, $\Omega_i$ be a set of $k_i$ integers, $\Omega=\prod_i\Omega_i$, and $f:\Omega\rightarrow\R$. Note that we do not require $k_i\geq 2$ for all $i$. If $k_j=1$ for some $j$, $f$ could be viewed as a function with one less variable, \ie as a function on $\prod_{i\neq j}\Omega_i$, but we still consider it as a function defined on $\prod_i\Omega_i$.

We shall say that $f$ is \emph{softly canalizing with respect to coordinate $i$ and $(a,b)\in\Omega_i\times\R$} if $f(x)=b$ whenever $x_i=a$, and simply that it is \emph{softly canalizing} if it is softly canalizing with respect to some $i,a,b$. \CORRECTION{We shall also say that \CORRECT{$F:\Omega\rightarrow\Omega$} is \emph{SNC} if all functions $F_i:\Omega\rightarrow\Omega_i$ are.}

Note that this definition differs slightly from the usual definition by the absence of condition on the values of $f$ for $x_i\neq a$: we do not require the existence of some $x$ such that $x_i\neq a$ and $f(x)\neq b$. In particular, constant functions are softly canalizing, though not canalizing.

If $f$ is canalizing with respect to $i,a,b$ and $k_i\geq 2$, we shall consider
$$
f\Restriction_{x_i\neq a}:\Omega\cap\{x\mid x_i\neq a\}\rightarrow\R,
$$
the restriction of $f$ to the set of $x\in\Omega$ such that $x_i\neq a$.

The class of softly nested canalizing on $\Omega=\prod_i\Omega_i$ is then defined by induction on the cardinality $|\Omega|=\prod_ik_i$ of $\Omega$. If $|\Omega|=1$, \ie $k_i=1$ for all $i$, any $f:\Omega\rightarrow\R$ is \emph{softly nested canalizing (SNC) on $\Omega$}. If $|\Omega|>1$, $f:\Omega\rightarrow\R$ is \emph{SNC on $\Omega$} if there exist $i,a,b$ such that
\begin{itemize}[itemsep=0pt]
\item $k_i\geq 2$,
\item $a$ is either the min or the max of $\Omega_i$,
\item $f$ is softly canalizing with respect to $i,a,b$,
\item $f\Restriction_{x_i\neq a}$ is SNC on $\Omega\cap\{x\mid x_i\neq a\}$, a strict subset of $\Omega$.
\end{itemize}

Intuitively, a function $f:\Omega\rightarrow\R$ is SNC if its domain $\Omega$ can be ``peeled'' by successively removing coordinate hyperplanes (defined by equations of the form $x_i=a$ with $a$ minimal or maximal) whose points are mapped by $f$ to the same value, whence the following characterization:

\begin{prop}
\label{prop:charac}
Letting $K=\sum_ik_i-n$, $f$ is SNC if and only if there exist a function $v:\{1,\ldots,K\}\rightarrow\{1,\ldots,n\}$, numbers $a_i\in\Omega_{v(i)}$ and $b_i\in\R$ for each $i\in\{1,\ldots,K\}$ such that:
$$
f(x)=
\begin{cases}
b_1 & \text{if } x_{v(1)}=a_1 \\
b_2 & \text{if } x_{v(1)}\neq a_1,x_{v(2)}=a_2 \\
\;\vdots & \quad\vdots \\
b_K & \text{if } x_{v(1)}\neq a_1,\ldots,x_{v(K-1)}\neq a_{K-1},x_{v(K)}=a_K \\
\end{cases}
$$
and for each $i\in\{1,\ldots,K\}$, $\Omega'_{v(i)}=\Omega_{v(i)}\setminus\{a_j \suchthat j<i, v(i)=v(j)\}$ is not a singleton and $a_i$ is either the min or the max of $\Omega'_{v(i)}$.
\end{prop}

In decomposing an NC function $f:(\Z/{k\Z})^n\rightarrow\Z/{k\Z}$, each coordinate $i\in\{1,\ldots,n\}$ is considered exactly once (in some order prescribed by a permutation $\sigma$) and the value of $f$ is fixed for $x_{\sigma(i)}$ in some segment $A_i$. This can be realized by successively fixing the value of $f$ for each $\alpha\in A_i$, and therefore, the class of SNC functions contains the class of NC functions, as stated in the following Proposition:

\begin{prop}
\label{prop:ncsnc}
If $f:(\Z/{k\Z})^n\rightarrow\Z/{k\Z}$ is NC, then it is SNC.
\end{prop}

\begin{proof}
Assume $f$ is NC with respect to $\sigma$, $A_1,\ldots,A_n$, $c_1,\ldots,c_{n+1}$. For each $i\in\{1,\ldots,n\}$, let
\begin{align*}
A_i &= \{\alpha_i^1,\ldots,\alpha_i^{|A_i|}\} \\
(\Z/{k\Z})\setminus A_i &= \{\alpha_i^{1+|A_i|},\ldots,\alpha_i^{k-1}\}.
\end{align*}
More precisely, the $K=n(k-1)$ numbers $\alpha_i^j\in\Z/{k\Z}$ are defined according to the following ordering which depends on the nature of the segment $A_i$:
\begin{itemize}
\item either $0\in A_i$, and then we choose the numbers $\alpha_i^j$ so that $0=\alpha_i^1<\cdots<\alpha_i^{|A_i|}$ and $k-1=\alpha_i^{1+|A_i|}>\cdots>\alpha_i^{k-1}$;
\item or $k-1\in A_i$, and we choose the $\alpha_i^j$ so that $k-1=\alpha_i^1>\cdots>\alpha_i^{|A_i|}$ and $0=\alpha_i^{1+|A_i|}<\cdots<\alpha_i^{k-1}$.
\end{itemize}
If both $0$ and $k-1\in A_i$, \ie if $A_i=\Z/k\Z$, we choose any of the two orders. Now, for each $i\in\{1,\ldots,n\}$ and $j\in\{1,\ldots,k-1\}$, let
$$
\beta_i^j=
\begin{cases}
c_i & \text{if } j\leq |A_i| \\
c_{n+1} & \text{otherwise.}
\end{cases}
$$
To comply with the characterization of SNC functions (Proposition \ref{prop:charac}), we relabel the numbers $\alpha_i^j,\beta_i^j$ by identifying the list
$$
\alpha_1^1,\ldots,\alpha_1^{|A_1|},\ldots,\alpha_n^1,\ldots,\alpha_n^{|A_i|},\alpha_1^{1+|A_1|},\ldots,\alpha_1^k,\ldots,\alpha_n^{1+|A_n|},\ldots,\alpha_n^{k-1}
$$
as the list $a_1,\ldots,a_K$, and by identifying similarly the list
$$
\beta_1^1,\ldots,\beta_1^{|A_1|},\ldots,\beta_n^1,\ldots,\beta_n^{|A_n|},\beta_1^{1+|A_1|},\ldots,\beta_1^k,\ldots,\beta_n^{1+|A_n|},\ldots,\beta_n^{k-1}
$$
as the list $b_1,\ldots,b_K$. Call $\phi$ this relabelling, which maps $r\in\{1,\ldots,K\}$ to the pair $\phi(r)=(i,j)$ such that $a_r=\alpha_i^j$ and $b_r=\beta_i^j$. For instance, $\phi(1)=(1,1)$ and $\phi(K)=(n,k-1)$.

Then finally, a function $v:\{1,\ldots,K\}\rightarrow\{1,\ldots,n\}$ is defined as follows: if $1\leq r\leq K$, there exists a unique pair $(i,j)$ such that $\phi(r)=(i,j)$, and we let $v(r)=\sigma(i)$. Then $f$ clearly satisfies the characterization of SNC functions, with the choice of function $v$ and numbers $a_r, b_r$.
\end{proof}

\subsection{Some examples}

\begin{itemize}
\item We have noticed \CORRECTION{in Section \ref{secsoft}} that constant functions from $(\Z/{k\Z})^n$ to $\Z/{k\Z}$ are not canalizing, therefore not NC, but they are SNC.
\item An easy induction on $k$ shows that the functions $\min$ and $\max:(\Z/{k\Z})^2\rightarrow\Z/{k\Z}$ are SNC. For instance, $\min=\min_k:\{0,\ldots,k-1\}^2\rightarrow\{0,\ldots,k-1\}$ is softly canalizing with respect to $1,0,0$, then $\min_k\Restriction_{x_1\neq 0}$ is softly canalizing with respect to $2,0,0$, and $\min_k\Restriction_{x_1\neq 0,x_2\neq 0}$ is identical to the function $\min_{k-1}:\{1,\ldots,k-1\}^2\rightarrow\{1,\ldots,k-1\}$, which is SNC.

However, they are not NC \cite{KLKAL17}. Intuitively, there is more freedom in the construction of SNC functions, which can be built by successively ``peeling'' coordinate hyperplanes defined on some coordinate $i$ (\ie by some equation $x_i=a$), then a coordinate hyperplane defined on some other coordinate $j$, and later a coordinate hyperplane defined on $i$ again.
\item For the same reason, the identity from $(\Z/{k\Z})^n$ to $(\Z/{k\Z})^n$ is SNC but not NC. As we shall see in Section \ref{examples}, this applies to the functions governing the regulation of genes such as \gene{Raf}, \gene{Dsor}, \gene{Drk}, \gene{Stat92E}.
\item In Section \ref{examples}, we shall see other examples of SNC and non SNC functions occurring in multivalued gene regulatory networks found in the literature.
\end{itemize}

\subsection{WNC multivalued functions and average sensitivity}
\label{sec:wnc}

It is possible to consider a slightly more general notion of multivalued canalization, by removing the condition on canalizing values in the definition of SNC functions. This is what we do in \cite{RR22}: \emph{weakly nested canalizing (WNC) functions} $f:(\Z/{k\Z})^n\rightarrow\Z/{k\Z}$ are defined like SNC functions, but the values $a$ used to define $f(x)$ for $x_i=a$ need not be extreme values (initially $0$ or $k-1$), they can be intermediate values: $0<a<k-1$.

By Proposition \ref{prop:ncsnc}, NC $\Rightarrow$ SNC $\Rightarrow$ WNC, and we prove in \cite{RR22} that WNC (hence NC and SNC) multivalued functions have ``low complexity'' in the sense that their \emph{average sensitivity} (see \cite[Chapter~8]{ODon14}) is bounded above by a constant (independent of $n$), while the average sensitivity of an arbitrary multivalued function is of order $\mathcal{O}(n)$.

It is worth noticing that the above implications are strict. For instance, the function $f:\Z/3\Z\times\Z/3\Z\rightarrow\Z/3\Z$ defined by the following table:
$$
f(x,y)=\quad
\begin{array}{l|ccc}
x\backslash y
&0&1&2 \\ \hline
0&2&0&0 \\
1&1&1&1 \\
2&2&0&2
\end{array}
$$
is WNC but not SNC: the first canalization has to take the intermediate value $x=1$ (and determines $f(x,y)=1$), then the second canalization takes $y=0$ (determining $f(x,y)=2$) or $y=1$ (determining $f(x,y)=0$), and finally $x=0$ or $x=2$. 

\section{Preservation of nested canalization under Bool\-ean mapping}
\label{sec:reduc}

\subsection{Mapping multivalued \CORRECTION{networks} to Boolean \CORRECTION{networks}}

Let $n$, $\Omega_i=\Z/k_i\Z$ and $\Omega=\prod_{i=1}^n\Omega_i$ be as above.

The injective mapping $\vh:\Omega\hookrightarrow\{0,1\}^k$, with $k=\sum_ik_i-n$, proposed by Van Ham \cite{vH79} is defined as follows. For $i\in\{1,\ldots,n\}$ and $1\leq a \leq k_i-1$, let $\vh_{i,a}:\Omega\rightarrow\{0,1\}$ be defined by
$$
\vh_{i,a}(x) =
\begin{cases}
1 & \text{if } x_i\geq a \\
0 & \text{otherwise,}
\end{cases}
$$
and
$$
\vh_i(x) = (\vh_{i,1}(x),\ldots,\vh_{i,k_i-1}(x)) = 1^{x_i}0^{k_i-1-x_i},
$$
where $1^\ell = (1,\ldots,1)$ is the $\ell$-tuple of $1$'s (and similarly for $0^\ell$), and tuples of $0$'s and $1$'s are represented by words (without commas and parentheses). The maps can be combined to define $\varphiVH:\Omega\hookrightarrow\{0,1\}^k$ by
\begin{align*}
\vh(x)
&= (\vh_{1,1}(x),\ldots,\vh_{1,k_1-1}(x),\ldots,\vh_{n,1}(x),\ldots,\vh_{n,k_n-1}(x))
\\
&= 1^{x_1}0^{k_1-1-x_1} \cdots 1^{x_n}0^{k_n-1-x_n}.
\end{align*}
The image $\vh(\Omega)$ of $\vh$ is a strict subset of $\{0,1\}^k$ (unless $k_i = 2$ for all $i$), and a point $x\in\{0,1\}^k$ is said to be \emph{admissible} when $x\in\vh(\Omega)$. Given \CORRECTION{a network} $F:\Omega\rightarrow\Omega$, by injectivity of $\vh$, \CORRECTION{the following commutative diagram}
$$
\begin{tikzcd}
\Omega \arrow{r}{F} \arrow[swap]{d}{\vh} & \Omega \arrow{d}{\vh} \\
\vh(\Omega) \arrow[swap]{r}{F^\vh} & \vh(\Omega)
\end{tikzcd}
$$
defines $F^\vh:\vh(\Omega)\rightarrow\vh(\Omega)$, the \emph{Booleanization} of $F$.

Clearly, if $F$ is Boolean (\ie $k_i=2$ for all $i$), then $\vh(\Omega)=\Omega=(\Z/2\Z)^n$ and $F^\vh=F$.

Let us mention that other injective mappings from $\Omega$ to $\{0,1\}^k$ may be defined but $\vh$ is the only one preserving neighbours and regulatory graphs \cite{DRC11,Ton19}:
\begin{itemize}
\item two neighbouring states \CORRECTION{(\ie states $x,y\in\Omega$ such that $dist(x,y)=\sum_{i=1}^n|x_i-y_i|=1$)} are mapped by $\vh$ to neighbouring states in $\vh(\Omega$),
\item the global regulatory graphs $\mathcal{G}(F)$ and $\mathcal{G}(F^\vh)$ underlying the dynamics $F$ and $F^\vh$ are isomorphic. We refer to, e.g., \cite{RR05,RR08b,Rue16} for the definition of $\mathcal{G}(F)$.
\end{itemize}
The purpose of the remainder of this section is to prove that $\vh$ maps SNC multivalued networks to NC (Boolean) networks. Before doing this in Theorem \ref{th:nc} we need to recall the definition of nested canalization for Boolean functions.

\subsection{Nested canalizing Boolean functions}

Let $k$ be a positive integer, $\sigma\in\Sig_k$ be a permutation, and $a_1,\ldots,a_k$, $b_1,\ldots,b_k\in\{0,1\}$. We recall that a Boolean function $f:\{0,1\}^k\rightarrow\{0,1\}$ is said to be \emph{nested canalizing (NC) with respect to $\sigma,a_1,\ldots,a_k,b_1,\ldots,b_k$} if 
\begin{equation}
\label{eq:nc}
\tag{$\ast$}
f(x)=
\begin{cases}
b_1 & \text{if } x_{\sigma(1)}= a_1 \\
b_2 & \text{if } x_{\sigma(1)}\neq a_1,x_{\sigma(2)}= a_2 \\
\;\vdots & \quad\vdots \\
b_k & \text{if } x_{\sigma(1)}\neq a_1,\ldots,x_{\sigma(k-1)}\neq a_{k-1},x_{\sigma(k)}= a_k.
\end{cases}
\end{equation}
Note that we slightly modify the usual definition by not requiring $f(x)=1-b_k$ when $x_{\sigma(1)}\neq a_1,\ldots,x_{\sigma(k-1)}\neq a_{k-1},x_{\sigma(k)}\neq a_k$.

We shall simply say that $f$ is \emph{NC} if it is NC with respect to some $\sigma,a_1,\ldots,a_k$, $b_1,\ldots,b_k$, and that a network $F=(F_1,\ldots,F_k):\{0,1\}^k\rightarrow\{0,1\}^k$ is \emph{NC} if $F_j:\{0,1\}^k\rightarrow\{0,1\}$ is NC for all $j$.

\subsection{Nested canalizing partial Boolean functions}
\label{sec:subdomains}

In Section \ref{sec:boolcanal}, we shall consider functions $f:X\rightarrow\{0,1\}$ defined on a subset $X\subseteq\{0,1\}^k$. We extend the notion of nested canalization to such partial functions, by simply saying that $f:X\rightarrow\{0,1\}$ is \emph{NC} with respect to the above data if condition (\ref{eq:nc}) holds for all $x\in X$.

\subsection{Relation between $\vh$ and nested canalization}
\label{sec:boolcanal}

\begin{thm}
\label{th:nc}
If $F=(F_1,\ldots,F_n):\Omega\rightarrow\Omega$ is SNC, then $F^\vh:\vh(\Omega)\rightarrow\vh(\Omega)$ is NC.
\end{thm}

\begin{proof}
As in the definition of $\vh$, we shall use double indices for Boolean functions, and let
\begin{align*}
F^\vh &= (F^\vh_{1,1},\ldots,F^\vh_{1,k_1-1},\ldots,F^\vh_{n,1},\ldots,F^\vh_{n,k_n-1}) \\
&= (F^\vh_1,\ldots,\ldots,F^\vh_n),
\end{align*}
where $F^\vh_{j,a}=\vh_{j,a}\circ F:\vh(\Omega)\rightarrow\{0,1\}$ and
$$
F^\vh_j=\vh_j\circ F=(F^\vh_{j,1},\ldots,F^\vh_{j,k_j-1}):\vh(\Omega)\rightarrow\{0,1\}^{k_j-1}.
$$
Therefore $F^\vh_j(\vh(x))=1^{F_j(x)}0^{k_j-1-F_j(x)}$ and $F^\vh_{j,a}(\vh(x)) = 1 \Leftrightarrow F_j(x)\geq a$.

Let us fix $j\in\{1,\ldots,n\}$ and $1\leq a \leq k_j-1$ and prove that $F^\vh_{j,a}$ is NC.

By assumption, $F_j:\Omega\rightarrow\Omega_j$ is SNC, hence by Proposition \ref{prop:charac}, there exist a function $v:\{1,\ldots,K\}\rightarrow\{1,\ldots,n\}$, with $K=\sum_ik_i-n$, and numbers $a_i\in\Omega_{v(i)}$ and $b_i\in\Omega_j$ for each $i\in\{1,\ldots,K\}$ such that, for all $x\in\Omega$:
$$
F_j(x)=
\begin{cases}
b_1 & \text{if } x\in\Omega^{(1)} \text{ and } x_{v(1)}=a_1 \\
b_2 & \text{if } x\in\Omega^{(2)}  \text{ and } x_{v(2)}=a_2 \\
\;\vdots & \quad\vdots \\
b_K & \text{if } x\in\Omega^{(K)} \text{ and } x_{v(K)}=a_K,
\end{cases}
$$
where for all $i\in\{1,\ldots,K\}$, $\Omega^{(i)}\subseteq\Omega$ is the set of $x$ such that
$$
x_{v(1)}\neq a_1,\ldots,x_{v(i-1)}\neq a_{i-1}.
$$
Note that $\Omega=\Omega^{(1)}\supset\Omega^{(2)}\supset\cdots\supset\Omega^{(K)}$ form a decreasing sequence of subsets of $\Omega$.

Moreover, by Proposition \ref{prop:charac}, the numbers $a_i$ are assumed to satisfy the following constraint:
$$
\text{either $a_i=\min\Omega^{(i)}_{v(i)}$ or $a_i=\max\Omega^{(i)}_{v(i)}$.}
$$
Let $\epsilon_i=0,a'_i=a_i+1$ in the first case (\emph{min} case) and $\epsilon_i=1,a'_i=a_i$ in the second case (\emph{max} case). Letting $y=\vh(x)$, we observe that in the min case, when $x\in\Omega^{(i)}$ we have
$$
x_{v(i)}=a_i \Leftrightarrow x_{v(i)}\leq a_i \Leftrightarrow y_{v(i),a_i+1}=\vh_{v(i),a_i+1}(x)=0,
$$
and that in the max case, we have
$$
x_{v(i)}=a_i \Leftrightarrow x_{v(i)}\geq a_i \Leftrightarrow y_{v(i),a_i}=\vh_{v(i),a_i}(x)=1.
$$
These two equivalences are summarized in the following property:
\begin{equation}
\label{eq:omega}
\tag{$P$}
x_{v(i)}=a_i \Leftrightarrow y_{v(i),a'_i}=\epsilon_i \text{ whenever } x\in\Omega^{(i)}.
\end{equation}
Let us prove by induction on $i$ that
\begin{equation}
\label{eq:induction}
\tag{$Q_i$}
x\in\Omega^{(i)}
\Leftrightarrow\;\;
\begin{cases}
y_{v(1),a'_1}\neq\epsilon_1 \\
\;\vdots \\
y_{v(i-1),a'_{i-1}}\neq\epsilon_{i-1}.
\end{cases}
\end{equation}
$Q_1$ is trivial because $\Omega^{(1)}=\Omega$, and by property $P$, we have
\begin{align*}
x\in\Omega^{(i+1)}
& \Leftrightarrow
\left(x\in\Omega^{(i)} \text{ and } x_{v(i)}\neq a_i\right)
\\
& \Leftrightarrow
\left(x\in\Omega^{(i)} \text{ and } y_{v(i),a'_i}\neq\epsilon_i\right),
\end{align*}
hence $Q_i$ entails $Q_{i+1}$, and we have shown that $Q_i$ holds for any $i$. Combining properties $P$ and $Q_i$, we then obtain:
$$
\begin{cases}
x\in\Omega^{(i)} \\
x_{v(i)}=a_i
\end{cases}
\Leftrightarrow\;\;
\begin{cases}
y_{v(1),a'_1}\neq\epsilon_1 \\
\;\vdots \\
y_{v(i-1),a'_{i-1}}\neq\epsilon_{i-1} \\
y_{v(i),a'_i}=\epsilon_i.
\end{cases}
$$
On the other hand, $F^\vh_{j,a}(y) = 1 \Leftrightarrow F_j(x)\geq a$. Hence, letting $\chi_{b\geq a}=1$ if $b\geq a$ and $0$ otherwise, we have $F^\vh_{j,a}(y) = \chi_{F_j(x)\geq a}$. In particular
$$
F_j(x)=b_i \Rightarrow F^\vh_{j,a}(y)=\chi_{b_i\geq a}.
$$
Therefore, the condition that $F_j$ is SNC gives, for any $y\in\vh(\Omega)$:
$$
F^\vh_{j,a}(y)=
\begin{cases}
\chi_{b_1\geq a} & \text{if } y_{v(1),a'_1}=\epsilon_1 \\
\chi_{b_2\geq a} & \text{if } y_{v(1),a'_1}\neq\epsilon_1,y_{v(2),a'_2}=\epsilon_2 \\
\;\vdots & \quad\vdots \\
\chi_{b_K\geq a} & \text{if } y_{v(1),a'_1}\neq\epsilon_1,\ldots,y_{v(K-1),a'_{K-1}}\neq\epsilon_{K-1},y_{v(K),a'_K}=\epsilon_K.
\end{cases}
$$
This means that $F^\vh_{j,a}$ is NC in the sense of Section \ref{sec:subdomains}. Since this holds for all $j,a$, $F^\vh$ is NC.
\end{proof}

\subsection{Counterexample for WNC functions}
\label{contrex}

Let $f:\Omega=\Z/3\Z\times\Z/3\Z\rightarrow\Z/3\Z$ be the function defined in Section \ref{sec:wnc}. It is WNC but not SNC, and $\vh(\Omega)\subseteq\{0,1\}^2\times\{0,1\}^2$. The function $f^\vh:\vh(\Omega)\rightarrow\{0,1\}^2$ is given by the following table:
$$
f^\vh(x,y)=\quad
\begin{array}{c|ccc}
&00&10&11 \\ \hline
00&11&00&00 \\
10&10&10&10 \\
11&11&00&11
\end{array}
$$
where pairs in $\{0,1\}^2$ are written without commas and parentheses. This table is simply obtained from the one defining $f$ by the substitutions $0\mapsto 00, 1\mapsto 10, 2\mapsto 11$.

Now, the two Boolean functions $f_1^\vh,f_2^\vh:\vh(\Omega)\rightarrow\{0,1\}$ such that $f^\vh=(f_1^\vh,f_2^\vh)$ are therefore:
$$
f_1^\vh(x,y)=\quad
\begin{array}{c|ccc}
&00&10&11 \\ \hline
00&1&0&0 \\
10&1&1&1 \\
11&1&0&1
\end{array}
\quad\quad
f_2^\vh(x,y)=\quad
\begin{array}{c|ccc}
&00&10&11 \\ \hline
00&1&0&0 \\
10&0&0&0 \\
11&1&0&1
\end{array}
$$
and it is easy to see that $f_1^\vh$ is NC but that $f_2^\vh$ is not NC. Therefore Theorem \ref{th:nc} does not extend to WNC functions.

\CORRECTION{
\section{Canalization of biological networks}
\label{examples}
}

Boolean NC functions appear predominantly in databases of Boolean genetic networks \cite{KBHKS20,Sub22}. If this canalizing property reflects physical characteristics of biological systems, we wondered whether this property is also present in multivalued functions that model these same systems. Thus, in order to have an insight into the characteristics of multivalued functions specified in the context of biological systems, \CORRECTION{we considered a sample of such functions from biological models.}

\CORRECT{
\subsection{Analysis of biological rules}
\label{study}
}
We selected the multivalued genes present in the logical models explored in \cite{Sub22} (4\CORRECTION{8} genes endowed with a ternary variable).

The logical rules corresponding to these genes are collected in the Appendix. To simplify the notation in the rules, the name of the gene is written to represent its activity variable (for instance $\gene{A}:1$ stands for $x_{\gene{A}} =1$), and the value of $F_{\gene{A}}$ is defined by means of the usual logical connectives (conjunction $\wedge$, disjunction $\vee$ and negation $\neg$). For example, the first and third lines of Table \ref{tab:MVrulesS1} mean:
$$
F_{\gene{Drk}}(x)=
\begin{cases}
1 & \text{if } x_{\gene{Der}}=1 \\
2 & \text{if } x_{\gene{Der}}=2 \\
0 & \text{otherwise}
\end{cases}
\qquad
F_{\gene{RI}}(x)=
\begin{cases}
1 & \text{if } x_{\gene{Dsor1}}=1 \text{ and } x_{\gene{Msk}}=1 \\
2 & \text{if } x_{\gene{Dsor1}}=2 \text{ and } x_{\gene{Msk}}=1 \\
0 & \text{otherwise}
\end{cases}
$$
This means that gene \gene{Drk} can be activated at level $1$ (resp. $2$) when its (unique) regulator \gene{Der} is at level $1$ (resp. $2$). Gene \gene{RI} has two regulators, \gene{Msk} (which is Boolean) and \gene{Dsor1}, and it can be activated at level $1$ (resp. $2$) if \gene{Msk} is present ($x_{\gene{Msk}}=1$) and \gene{Dsor1} is at level $1$ (resp. $2$).

Some situations are more complex. For instance, gene \gene{MadMed} (extracted from the model described in \cite{Mbo13}) has three regulators: a multilevel activator \gene{Tkv}, a Boolean activator \gene{Sax} and a Boolean inhibitor \gene{Dad}. The rule for \gene{MadMed} is the following (cf Table \ref{tab:MVrulesC1}):
$$
F_{\gene{MadMed}}(x)=
\begin{cases}
1 & \text{if } (x_{\gene{Tkv}}=1 \text{ or } x_{\gene{Sax}}=1) \text{ and } x_{\gene{Dad}}=0 \text{ and } x_{\gene{Tkv}} < 2 \\
2 & \text{if } x_{\gene{Tkv}}=2 \text{ and } x_{\gene{Dad}}=0\\
0 & \text{otherwise.}
\end{cases}
$$
Thus, in the absence of the inhibitor \gene{Dad}, the presence of one of the two activators at level 1 (max level for \gene{Sax} and intermediate for \gene{Tkv}) allows \gene{MadMed} to be activated at its level 1. \gene{MadMed} can be activated at its maximum level as soon as its activator \gene{Tkv} is at its maximum level $2$ and the inhibitor \gene{Dad} is absent. 

We studied the canalizing properties (NC, SNC, WNC) of the logical functions. We also applied the Boolean mapping $\vh$ and studied the canalizing property of the resulting Boolean functions. The results of the analysis are collected in Tables \ref{tab:MVnodesA} to \ref{tab:MVnodesD}, reflecting four different qualitative situations:
\begin{enumerate}[label=(\alph*)]
\item \label{situA}
functions which are NC (hence SNC, WNC, and with NC Boolean\-iz\-ation $F^\vh$);
\item \label{situB}
functions which are not NC, but SNC (hence WNC, and with NC Booleanization);
\item \label{situC}
functions which are not WNC (hence neither NC nor SNC) but with NC Booleanization;
\item \label{situD}
functions which are not WNC and with non NC Booleanizations.
\end{enumerate}
Although we noticed that WNC are in general not necessarily SNC, remark that in the above set of functions, all WNC functions are actually SNC.

\CORRECTION{
\subsection{Classification of multivalued functions based on their structure}
\label{structureS}
}
The structure of multivalued logical rules varies greatly. A quite frequent situation is when the level of only one regulator determines the target value of the gene:
\begin{equation}
\label{C1}
\tag{$S$}
\begin{cases}
\gene{A}=1 & \mbox{if } \gene{B}\wedge\phi \\
\gene{A}=2 & \mbox{if } \neg\gene{B}\wedge\phi
\end{cases}
\end{equation}
with \gene{B} a regulator of \gene{A}, and $\phi$ a ``context'' (conditions on the presence or absence of other regulators). This concerns the genes listed in Tables \ref{tab:MVrulesS1}, \ref{tab:MVrulesS2}, \ref{tab:MVrulesS3}. We can see in Tables \ref{tab:MVnodesA} to \ref{tab:MVnodesD} that the rules satisfying this structure (\ref{C1}) generally have strong canalizing properties: situations \ref{situA} or \ref{situB}. In particular, their Booleanizations are all NC. Then, depending on the context, we distinguish three situations:
\begin{itemize}
\item if $\phi$ depends on multivalued variables and is expressed only with $\wedge$, then the function behaves according to situation \ref{situB}; this is the case of, e.g., \gene{Stat92E}; these functions are listed in Table \ref{tab:MVrulesS1};
\item if $\phi$ is expressed with at least one $\vee$, then the function behaves according to situation \ref{situC} (functions listed in Table \ref{tab:MVrulesS2});
\item if $\phi$ depends on Boolean variables and is expressed only with $\wedge$, then the function behaves according to situation \ref{situA} (functions listed in Table \ref{tab:MVrulesS3}). Note that nodes \gene{IL4RA} and \gene{E2F3} are controlled by a multivalued gene (resp. \gene{STAT5} and \gene{CHEK}), but their level depends only on one threshold (to be above or below level $2$).
\end{itemize}

\begin{table}
\begin{center}
\begin{tabular}{|l|lcccc|}
\hline
Genes in situation \ref{situA} & NC & SNC & WNC & Bool. NC & Struct. \\
\hline
\gene{E\_Spl} & Yes & Yes & Yes & Yes &
\\*
\gene{mQH2\_Q} & Yes & Yes & Yes & Yes & (\ref{C1})
\\*
\gene{mdH} & Yes & Yes & Yes & Yes &  (\ref{C1})
\\*
\gene{mGR} & Yes & Yes & Yes & Yes &
\\*
\gene{mGSH\_GSSG} & Yes & Yes & Yes & Yes &  (\ref{C1})
\\*
\gene{mTRX} & Yes & Yes & Yes & Yes &  (\ref{C1})
\\*
\gene{cGSH\_GSSG} & Yes & Yes & Yes & Yes &  (\ref{C1})
\\*
\gene{cGR} & Yes & Yes & Yes & Yes &
\\*
\gene{E2F3} & Yes & Yes & Yes & Yes &  (\ref{C1})
\\*
\gene{IL12RB1} & Yes & Yes & Yes & Yes &  (\ref{C1})
\\*
\gene{IL4RA} & Yes & Yes & Yes & Yes &  (\ref{C1})
\\*
\gene{\CORRECTION{ATM}} & Yes & Yes & Yes & Yes &  (\ref{C1})
\\*
\gene{\CORRECTION{CHEK12}} & Yes & Yes & Yes & Yes &  (\ref{C1})
\\\hline
\end{tabular}
\caption{\label{tab:MVnodesA} Multivalued genes in situation \ref{situA} occurring in logical models from published gene regulatory network models considered in \cite{Sub22}. The rightmost columns indicate whether the Booleanized functions are nested canalizing, and which rules satisfy structure (\ref{C1}).}
\end{center}
\end{table}

\begin{table}
\begin{center}
\begin{tabular}{|l|lcccc|}
\hline
Genes in situation \ref{situB} & NC & SNC & WNC & Bool. NC & Struct. \\
\hline
\gene{Drk} & No & Yes & Yes & Yes &  (\ref{C1})
\\*
\gene{Dsor1} & No & Yes & Yes & Yes &  (\ref{C1})
\\*
\gene{Pnt} & No & Yes & Yes & Yes &  (\ref{C1})
\\*
\gene{Stat92E} & No & Yes & Yes & Yes &  (\ref{C1})
\\*
\gene{Raf} & No & Yes & Yes & Yes &  (\ref{C1})
\\*
\gene{RI} & No & Yes & Yes & Yes &  (\ref{C1})
\\*
\gene{Sos} & No & Yes & Yes & Yes &  (\ref{C1})
\\*
\gene{Tkv} & No & Yes & Yes & Yes &
\\*
\gene{mNNT} & No & Yes & Yes & Yes &
\\*
\gene{mCa} & No & Yes & Yes & Yes &  (\ref{C1})
\\*
\gene{mGPX} & No & Yes & Yes & Yes &  (\ref{C1})
\\*
\gene{mTR} & No & Yes & Yes & Yes &
\\*
\gene{cGPX} & No & Yes & Yes & Yes &  (\ref{C1})
\\*
\gene{cTR} & No & Yes & Yes & Yes &
\\*
\gene{cTRX} & No & Yes & Yes & Yes &  (\ref{C1})
\\*
\gene{STAT5} & No & Yes & Yes & Yes &
\\*
\gene{\CORRECTION{Twi}} & No & Yes & Yes & Yes &  (\ref{C1})
\\\hline
\end{tabular}
\caption{\label{tab:MVnodesB} Multivalued genes in situation \ref{situB} occurring in logical models from \cite{Sub22}.}
\end{center}
\end{table}

\begin{table}
\begin{center}\begin{tabular}{|l|lcccc|}
\hline
Genes in situation \ref{situC} & NC & SNC & WNC & Bool. NC & Struct. \\
\hline
\gene{Ras} & No & No & No & Yes &
\\*
\gene{MadMed} & No & No & No & Yes &
\\*
\gene{Hop} & No & No & No & Yes &
\\*
\gene{mNADPH\_NADP} & No & No & No & Yes &
\\*
\gene{mNADH\_NAD} & No & No & No & Yes &  (\ref{C1})
\\*
\gene{cCa} & No & No & No & Yes  &
\\*
\gene{KrebsCycle} & No & No & No & Yes &  (\ref{C1})
\\*
\gene{VIM}& No & No & No & Yes &
\\*
\gene{CDH1} & No & No & No & Yes &
\\*
\gene{EMT} & No & No & No & Yes &
\\*
\gene{IL2R}& No & No & No & Yes &  (\ref{C1})
\\*
\gene{IL4R} & No & No & No & Yes &  (\ref{C1})
\\\hline
\end{tabular}
\caption{\label{tab:MVnodesC} Multivalued genes in situation \ref{situC} occurring in logical models from \cite{Sub22}.}
\end{center}
\end{table}

\begin{table}
\begin{center}
\begin{tabular}{|l|lcccc|}
\hline
Genes in situation \ref{situD} & NC & SNC & WNC & Bool. NC & Struct. \\
\hline
\gene{mROS} & No & No & No & No  &
\\*
\gene{Der} & No & No & No & No  &
\\*
\gene{cROS} & No & No & No & No  &
\\*
\gene{cNADPH\_NADP} & No & No & No & No  &
\\*
\gene{E2F1}& No & No & No & No  &
\\*
\gene{Spi1} & No & No & No & No  &
\\\hline
\end{tabular}
\caption{\label{tab:MVnodesD} Multivalued genes in situation \ref{situD} occurring in logical models from \cite{Sub22}.}
\end{center}
\end{table}

The rules that do not satisfy structure (\ref{C1}) are collected in Tables \ref{tab:MVrulesC1} and \ref{tab:MVrulesC2}. In particular, Table \ref{tab:MVrulesC2} gathers all the rules in situation \ref{situD}, for which even the Booleanized functions are not NC.

\CORRECTION{
\subsection{Enrichment in SNC functions}
\label{sec:stats}

To get a flavour of the canalizing properties of biological functions, we \CORRECT{investigated} whether the distribution of the 48 multivalued rules that we considered (reported in Tables \ref{tab:MVnodesA} to \ref{tab:MVnodesD}) is as expected or shifted towards more than expected canalization.

Deriving a formula to calculate the exact number of SNC or WNC functions (similar to the formula given in \cite{ML12} for the number of NC multivalued functions) is not trivial and is beyond the scope of this work. Here, we proposed a method for estimating an upper bound on the maximum number of SNC functions (\ref{comment}), and elaborated a python program that calculates it (\ref{python}).

Thus, we calculated, for each value of arity $n\in\{1,\ldots,9\}$, 
\begin{itemize}
\item the proportion of SNC functions among our functions of arity $n$ (Tables \ref{tab:MVnodesA} to \ref{tab:MVnodesD}), \ie functions from $\prod_{i=1}^n\Omega_i$ to $\Z/3\Z$, where $\Omega_i=\Z/2\Z$ or $\Z/3\Z$; 
\item the upper bound on the expected proportion of SNC functions among all functions of arity $n$. 
\end{itemize}

We were able to calculate the upper bound on the number of SNC functions only for $n\leq 5$ (computation for $n=6$ reaches the precision limit of floating point numbers). The results are reported in Table \ref{tab:prop}.
\begin{table}[ht]
$$
\setlength{\arraycolsep}{10pt}
\begin{array}{ccc}
\multirow{2}{*}{$n$} & \text{Proportion of SNC} & \text{Theoretical proportion}
\\
& \text{functions in Tables \ref{tab:MVnodesA} to \ref{tab:MVnodesD}} & \text{of SNC functions}
\\\hline
1 & 7/7 & 1
\\
2 & 16/20 & \leq 0.83
\\
3 & 3/6 & \leq 8.5\cdot 10^{-7}
\\
4 & 1/5 & \leq 10^{-29}
\\
5 & 0/2 & \leq 5.6\cdot 10^{-104}
\\
6 & 2/5 & \leq 5.6\cdot 10^{-104}
\\
7 & 0/1 & \leq 5.6\cdot 10^{-104}
\\
8 & 1/1 & \leq 5.6\cdot 10^{-104}
\\
9 & 0/1 & \leq 5.6\cdot 10^{-104}
\end{array}
$$
\CORRECTION{\caption{\label{tab:prop} Proportion of SNC functions among multivalued functions of the data set (second column) depending on the arity $n$ (first column). For instance, $3$ out of the $6$ functions of arity $3$ are SNC. The third column presents the upper bound on the expected proportion ($8.5\cdot 10^{-7}$ for $n=3$).}}
\end{table}

\CORRECT{Clearly, the proportion of SNC decreases as $n$ increases, both in the dataset and in the theoretical upper bound.} Of course, the size of this dataset does not allow us to draw statistically significant conclusions. \CORRECT{Nevertheless, we can draw some observations: for $n=2$, $80\%$ of functions from our dataset are SNC, while the upper bound is $83\%$, \textit{i.e.}, no noteworthy difference; however, for $n=3$, $4$, $6$ and $8$, the observed proportion of SNC, albeit based on small sample size, differs substantially from the expected. Thus, while the ``signal'' from this sample suggests that the proportion of SNC functions is rather high, a study on a larger dataset would be required to deepen this analysis.}
}

\CORRECTION{\section{Concluding remarks and prospects}
\label{sec:concl}

We prove in this manuscript that Booleanization preserves the SNC (and therefore the NC) property (Theorem \ref{th:nc}). On the other hand, the preliminary analysis of the canalizing property of multivalued biological functions (section \ref{sec:stats}) suggests a certain relevance of canalization, notably SNC, among models arising from gene networks, as is the case for Boolean biological functions. It is worth observing that within the study sample, SNC and WNC properties are not distinguished. The significant number of multivalued functions that are not WNC, but whose Booleanizations are NC (gathered in Table \ref{tab:MVnodesC}), calls for further study.

Importantly, it should be noted that we have not specified the assumptions for updating the system while they play an important role in the dynamical behavior and its properties. Implicitly, we considered here that the genes directly reach their target values: for instance, \gene{Drk} can be activated at level $2$ when \gene{Der} is at level $2$, and if the current state of \gene{Drk} is $0$, it will update directly to value $2$. However, in the context of biological systems, this is not a realistic hypothesis, steps of $1$ are more common: in the previous example, if \gene{Der} is at level $2$ and \gene{Drk} is not active, its activity will first reach level $1$. The introduction of the ``steps of $1$'' hypothesis in the definition of the function would obviously change the canalizing properties.
}

\CORRECT{
\section*{Code availability}

The python code used in this study is available at

\begin{center}
\url{https://www.irif.fr/\~ruet/}
\end{center}
}

\bibliographystyle{plain}

\begin{thebibliography}{10}
%
\bibitem{Col17}
S. Collombet, C. van Oevelen, J. L. Sardina Ortega, W. Abou-Jaoudé, B. Di Stefano, M. Thomas-Chollier, T. Graf, and  D. Thieffry.
\newblock Logical modeling of lymphoid and myeloid cell specification and transdifferentiation.
\newblock {\em Proc. Natl. Acad. Sci.}, 114(23):5792--5799, 2017.
%
\bibitem{DRC11}
G. Didier, É. Remy, and C. Chaouiya.
\newblock Mapping multivalued onto Boolean dynamics.
\newblock {\em J. Theoret. Biol.}, 270(1):177--184, 2011.
%
\bibitem{KBHKS20}
\CORRECTION{C. Kadelka, T.-M. Butrie, E. Hilton, J. Kinseth, and H. Serdarevic.}
\newblock A meta-analysis of Boolean network models reveals design principles of gene regulatory networks.
\newblock {\em arXiv:2009.01216}, 2020.
%
\bibitem{KLKAL17}
C. Kadelka, Y. Li, J. Kuipers, J. O. Adeyeye, and R. Laubenbacher.
\newblock Multistate nested canalizing functions and their networks.
\newblock {\em Theoret. Comput. Sci.}, 675:1--14, 2017.
%
\bibitem{Kau93}
S. A. Kauffman.
\newblock {\em The origins of order: Self organization and selection in evolution}.
\newblock Oxford University Press, 1993.
%
\bibitem{Kau03}
S. Kauffman, C. Peterson, B. Samuelsson, and C. Troein.
\newblock Random Boolean network models and the yeast transcriptional network.
\newblock {\em Proc. Natl. Acad. Sci.}, 100(25), 2003.
%
\bibitem{Kau04}
S. Kauffman, C. Peterson, B. Samuelsson, and C. Troein.
\newblock Genetic networks with canalyzing Boolean rules are always stable.
\newblock {\em Proc. Natl. Acad. Sci.}, 101(49), 2004.
%
\bibitem{Klo13}
J. G. Klotz, R. Heckel, and S. Schober.
\newblock Bounds on the average sensitivity of nested canalizing functions.
\newblock {\em Plos One}, 8(5), 2013.
%
\bibitem{Lau13}
Y. Li, J. O. Adeyeye, D. Murrugarra, B. Aguilar, and R. Laubenbacher.
\newblock Boolean nested canalizing functions: A comprehensive analysis.
\newblock {\em Theoret. Comput. Sci.}, 481:24--36, 2013.
%
\bibitem{Mbo13}
A. Mbodj, G. Junion, C. Brun, E. Furlong, and D. Thieffry.
\newblock Logical modelling of Drosophila signalling pathways.
\newblock {\em Mol. Biosyst.}, 9(9):2248--58, 2013.
%
\bibitem{ML11}
D. Murrugarra and R. Laubenbacher.
\newblock Regulatory patterns in molecular interaction networks.
\newblock {\em J. Theor. Biol.}, 288:66--72, 2011.
%
\bibitem{ML12}
D. Murrugarra and R. Laubenbacher.
\newblock The number of multistate nested canalyzing functions.
\newblock {\em Physica D: Nonlinear Phenomena}, 241(10):929--938, 2012.
%
\bibitem{Nal10}
A. Naldi, J. Carneiro, C. Chaouiya, and D. Thieffry.
\newblock Diversity and plasticity of Th cell types predicted from regulatory network modelling.
\newblock {\em PLoS Comput. Bio.}, 6(9), 2010.
%
\bibitem{ODon14}
R. O'Donnell.
\newblock {\em Analysis of Boolean functions}.
\newblock Cambridge University Press, 2014.
%
\bibitem{Rem15}
É. Remy, S. Rebouissou, C. Chaouiya, A. Zinovyev, F. Radvany, and L. Calzone.
\newblock A modeling approach to explain mutually exclusive and co-occurring genetic alterations in bladder tumorigenesis.
\newblock  {\em Cancer Res.}, 75(19): 4042--4052, 2015.
%
\bibitem{RR05}
É. Remy and P. Ruet.
\newblock On differentiation and homeostatic behaviours of Boolean dynamical systems.
\newblock BioConcur 2005, in {\em Transactions on Computational Systems Biology VIII}, Springer LNCS 4780: 92--101, 2007.
%
\bibitem{RR22}
É. Remy and P. Ruet.
\newblock Average sensitivity of nested canalizing multivalued functions.
\newblock {\em Computational Methods in Systems Biology XXI}, Springer LNCS 14137, 2023.
%
\bibitem{RR08b}
É. Remy, P. Ruet, and D. Thieffry.
\newblock Graphic requirements for multistability and attractive cycles in a Boolean dynamical framework.
\newblock {\em Adv. Appl. Math.}, 41(3):335--350, 2008.
%
\bibitem{Rue16}
P. Ruet.
\newblock Local cycles and dynamical properties of Boolean networks.
\newblock {\em Math. Struct. Comput. Sci.}, 26(4):702--718, 2016.
%
\bibitem{San19}
J.A. Sánchez-Villanueva, O. Rodríguez-Jorge, O. Ramírez-Pliego, G.R. Salgado, W. Abou-Jaoudé, C. Hernandez, A. Naldi, D. Thieffry, and M.A. Santana.
\newblock Contribution of ROS and metabolic status to neonatal and adult CD8+ T cell activation.
\newblock {\em PLoS ONE} 14(12): e0226388, 2019.
%
\bibitem{Sil20}
D.A. Silveira, and J.C. Mombach.
\newblock Dynamics of the feedback loops required for the phenotypic stabilization in the epithelial-mesenchymal transition.
\newblock {\em FEBS J.}, 287(3):578--588, 2020.
%
\bibitem{Sub22}
A. Subbaroyan, O. C. Martin, and A. Samal.
\newblock Minimum complexity drives regulatory logic in Boolean models of living systems.
\newblock {\em PNAS Nexus}, 1:1--12, 2022.
%
\bibitem{Tho73}
R. Thomas.
\newblock Boolean formalization of genetic control circuits.
\newblock {\em J. Theor. Biol.}, 42:563--585, 1973.
%
\bibitem{Tho91}
R. Thomas.
\newblock Regulatory networks seen as asynchronous automata: a logical description.
\newblock {\em J. Theor. Biol.}, 153:1--23, 1991.
%
\bibitem{Ton19}
E. Tonello.
\newblock On the conversion of multivalued gene regulatory networks to Boolean dynamics.
\newblock {\em Discrete Appl. Math.}, 259(C):193--204, 2019.
%
\bibitem{vH79}
P. Van Ham.
\newblock How to deal with variables with more than two levels.
\newblock In {\em Kinetic logic: a Boolean approach to the analysis of complex regulatory systems}. Springer, 326--343, 1979.
%
\bibitem{Wad42}
C.H. Waddington. 
\newblock Canalization of development and the inheritance of acquired characters. 
\newblock{\em Nature}, 150:563–565, 1942.
\end{thebibliography}

\newpage

\appendix
\section{Logical rules}
\label{appendix}

\small

\begin{longtable}{|L|L|L|L|L|}
\hline
\text{Gene } i & x_i & \text{Rule} & \text{Ref}
\\\hline \endfirsthead
\hline
\text{Gene } i & x_i & \text{Rule \hspace{3cm} (continued)} & \text{Ref}
\\\hline \endhead
\multirow{2}{*}{\gene{IL12RB1}}
& 1 & \neg \gene{IRF1}& \multirow{4}{*}{\cite{Nal10}}
\\*
& 2 &\gene{IRF1}&
\\\cline{1-3}
\multirow{2}{*}{\gene{IL4RA}}
& 1 & \neg \gene{STAT5}:2& 
\\*
& 2 &\gene{STAT5}:2&
\\\hline
\gene{mGSH}
& 1 & \gene{mGR} \wedge \gene{mGPX}& \multirow{10}{*}{\cite{San19}}
\\*
\quad \gene{\_GSSG}
& 2 & \gene{mGR} \wedge \neg \gene{mGPX}&
\\\cline{1-3}
\multirow{2}{*}{\gene{mdH}}
& 1 & \gene{ETC} \wedge \gene{ATPSyn}&
\\*
& 2 & \gene{ETC} \wedge \neg \gene{ATPSyn}&
\\\cline{1-3}
\multirow{2}{*}{\gene{mQH2\_Q}}
& 1 & \gene{ETC} \wedge \neg \gene{mdH}&
\\*
& 2 & \gene{ETC} \wedge \gene{mdH}&
\\\cline{1-3}
\multirow{2}{*}{\gene{mTRX}}
& 1 & \gene{mTR} \wedge \gene{mROS}&
\\*
& 2 & \gene{mTR} \wedge \neg \gene{mROS}&
\\\cline{1-3}
\gene{cGSH}
& 1 & \gene{cGR} \wedge \gene{cGPX}&
\\*
\quad \gene{\_GSSG}
& 2 & \gene{cGR} \wedge \neg \gene{cGPX}&
\\\hline
\multirow{2}{*}{\gene{E2F3}}
& 1 & \neg \gene{CHEK}:2 \wedge \neg \gene{RB1} \wedge \gene{RAS}&\multirow{6}{*}{\cite{Rem15}}
\\*
& 2 &  \gene{CHEK}:2 \wedge \neg \gene{RB1} \wedge \gene{RAS}&
\\\cline{1-3}
\multirow{2}{*}{\gene{\CORRECTION{ATM}}}
& 1 & \neg \gene{DNAdam} \wedge \neg \gene{E2F1} &
\\*
& 2 &  \gene{DNAdam} \wedge \gene{E2F1} &
\\\cline{1-3}
\multirow{2}{*}{\gene{\CORRECTION{CHEK}}}
& 1 & \neg \gene{ATM} \wedge \neg \gene{E2F1} &
\\*
& 2 &  \gene{ATM} \wedge \gene{E2F1} &
\\\hline
\caption{\label{tab:MVrulesS3} \CORRECTION{Logical rules associated to multivalued genes from published gene regulatory network models considered in \cite{Sub22}, which are NC and satisfy structure (\ref{C1}) (situation \ref{situA}).}}
\end{longtable}

\newpage

\vspace*{-2cm}
\begin{longtable}{|L|L|L|L|L|}
\hline
\text{Gene } i & x_i & \text{Rule} & \text{Ref}
\\\hline \endfirsthead
\hline
\text{Gene } i & x_i & \text{Rule \hspace{3cm} (continued)} & \text{Ref}
\\\hline \endhead
\multirow{2}{*}{\gene{Drk}}
& 1 & \gene{Der}:1
& \multirow{17}{*}{\cite{Mbo13}}
\\*
& 2 &\gene{Der}:2
&
\\\cline{1-3}
\multirow{2}{*}{\gene{Dsor1}}
& 1 & \gene{Raf}:1& 
\\*
& 2 &\gene{Raf}:2&
\\\cline{1-3}
\multirow{2}{*}{\gene{RI}}
& 1 & \gene{Dsor1}:1  \wedge  \gene{Msk}& 
\\*
& 2 &\gene{Dsor1}:2  \wedge  \gene{Msk}&
\\\cline{1-3}
\multirow{2}{*}{\gene{Pnt}}
& 1 & \gene{RI}:1&
\\*
& 2 &\gene{RI}:2&
\\\cline{1-3}
\multirow{2}{*}{\gene{Sos}}
& 1 & \gene{Drk}:1&
\\*
& 2 &\gene{Drk}:2&
\\\cline{1-3}
\multirow{3}{*}{\gene{Stat92E}}
& 1 & \gene{Hop}:1 \wedge \phi& 
\\*
& 2 & \gene{Hop}:2 \wedge \phi \ \text{where } \phi = \neg \gene{Su_{var}} \wedge \neg \gene{Ptp61F} \wedge \phantom{0}&
\\*
&& \phantom{\gene{Hop}:2 \wedge \phi \ } \neg \gene{Ken} \wedge \neg \gene{Brwd3} \wedge \neg \gene{Socs44A}&
\\\cline{1-3}
\multirow{2}{*}{\gene{Raf}}
& 1 & \gene{Ras}:1 \wedge\gene{Cnk} \wedge\gene{Src42} \wedge\gene{Ksr}&
\\*
& 2 & \gene{Ras}:2 \wedge\gene{Cnk} \wedge\gene{Src42} \wedge\gene{Ksr}&
\\\cline{1-3}
\multirow{2}{*}{\gene{\CORRECTION{Twi}}}
& 1 & \gene{Da} \wedge \gene{Emc} \wedge \gene{E\_Spl}:1&
\\*
& 2 & (\gene{Da} \wedge \neg\gene{Emc}) \vee   (\gene{Da} \wedge \gene{Emc}\wedge \neg\gene{E\_Spl})&
\\\hline
\multirow{2}{*}{\gene{mCa}}
& 1 & \gene{cCa}:1 & \multirow{8}{*}{\cite{San19}}
\\*
& 2 & \gene{cCa}:2 &
\\\cline{1-3}
\multirow{2}{*}{\gene{mGPX}}
& 1 & \gene{mGSH\_GSSG}:1 \wedge \neg \gene{mROS}&
\\*
& 2 & \gene{mGSH\_GSSG}:2 \wedge \neg \gene{mROS}&
\\\cline{1-3}
\multirow{2}{*}{\gene{cTRX}}
& 1 & \gene{cTR} \wedge \gene{cROS}:1&
\\*
& 2 & \gene{cTR} \wedge \neg \gene{cROS}&
\\\cline{1-3}
\multirow{2}{*}{\gene{cGPX}}
& 1 & \gene{cGSH\_GSSG}:1 \wedge \neg \gene{cROS}&
\\*
& 2 & \gene{cGSH\_GSSG}:2 \wedge \neg \gene{cROS}&
\\\hline

\caption{\CORRECTION{\label{tab:MVrulesS1} Logical rules associated to multivalued genes from published gene regulatory network models considered in \cite{Sub22}, which are SNC but not NC, and satisfy structure (\ref{C1}) (situation \ref{situB}).}}
\end{longtable}



\begin{longtable}{|L|L|L|L|L|}
\hline
\text{Gene } i & x_i & \text{Rule} & \text{Ref}
\\\hline \endfirsthead
\hline
\text{Gene } i & x_i & \text{Rule \hspace{3cm} (continued)} & \text{Ref}
\\\hline \endhead
\gene{mNADH}
& 1 & (\gene{KrebsCycle} \vee \gene{mPDH} \vee \gene{FAO}) \wedge \gene{ETC}& \multirow{6}{*}{\cite{San19}}
\\*
\quad \gene{\_NAD}
& 2 & (\gene{KrebsCycle} \vee \gene{mPDH} \vee \gene{FAO}) \wedge \neg \gene{ETC}&
\\\cline{1-3}
\multirow{4}{*}{\gene{KrebsCycle}}
& 1 & \neg \gene{mCa} \wedge \neg \gene{mNADH\_NAD} \wedge \phantom{0}&
\\*
&& \qquad (\gene{AcetylCoA} \vee \gene{GLUTAMINOLYSIS})&
\\*
& 2 & \gene{mCa} \wedge \neg \gene{mNADH\_NAD} \wedge \phantom{0}&
\\*
&& \qquad (\gene{AcetylCoA} \vee \gene{GLUTAMINOLYSIS})&
\\\hline
\multirow{2}{*}{\gene{IL2R}}
& 1 & \neg \gene{IL2RA} \wedge \gene{CGC} \wedge \gene{IL2RB} \wedge (\gene{IL2} \vee \gene{IL2\_e}) & \multirow{4}{*}{\cite{Nal10}}
\\*
& 2 & \gene{IL2RA} \wedge \gene{CGC} \wedge \gene{IL2RB} \wedge  (\gene{IL2} \vee \gene{IL2\_e})&
\\\cline{1-3}
\multirow{2}{*}{\gene{IL4R}}
& 1 & \gene{IL4RA}:1 \wedge \gene{CGC} \wedge  (\gene{IL4} \vee \gene{IL4\_e})& 
\\*
& 2 &  \gene{IL4RA}:2 \wedge \gene{CGC} \wedge (\gene{IL4} \vee \gene{IL4\_e})&
\\\hline
\caption{\CORRECTION{\label{tab:MVrulesS2} Logical rules associated to multivalued genes from published gene regulatory network models considered in \cite{Sub22}, which are not WNC but with NC Booleanization, and satisfy structure (\ref{C1}) (situation \ref{situC}).}}
\end{longtable}

\newpage

\vspace*{-4cm}
\begin{longtable}{|L|L|L|L|L|}
\hline
\text{Gene } i & x_i & \text{Rule} & \text{Ref}
\\\hline \endfirsthead
\hline
\text{Gene } i & x_i & \text{Rule \hspace{3cm} (continued)} & \text{Ref}
\\\hline \endhead
\multirow{4}{*}{\gene{mNNT}}
& 1 & \gene{mNADH\_NAD}:1 \wedge \gene{mdH}:1 & \multirow{27}{*}{\cite{San19}}
\\*
& 2 & (\gene{mNADH\_NAD}:2 \wedge \gene{mdH}:1) \vee \phantom{0} &
\\*
&& \qquad (\gene{mNADH\_NAD}:2 \wedge\gene{mdH}:2) \vee \phantom{0} &
\\*
&& \qquad (\gene{mNADH\_NAD}:1 \wedge\gene{mdH}:2) &
\\\cline{1-3}
\gene{mNADPH}
& 1 &  (\neg \gene{mNNT} \wedge\gene{mIDH2} ) \vee (\gene{mNNT}:1 \wedge\neg \gene{mIDH2} )&
\\*
\quad \gene{\_NADP}
& 2 & \gene{mNNT}:2 \vee (\gene{mNNT}:1  \wedge\gene{mIDH2}) &
\\\cline{1-3}
\multirow{4}{*}{\gene{mGR}}
& 1 & (\gene{mNADPH\_NADP}:1 \vee \gene{mNADPH\_NADP}:2) \wedge \phantom{0} &
\\*
&& \qquad (\neg\gene{mGSH\_GSSG} \vee \gene{mGSH\_GSSG}:1) &
\\*
& 2 & (\gene{mNADPH\_NADP}:1 \vee \gene{mNADPH\_NADP}:2) \wedge \phantom{0}&
\\*
&& \qquad \gene{mGSH\_GSSG}:2&
\\\cline{1-3}
\multirow{3}{*}{\gene{mTR}}
& 1 & (\gene{mNADPH\_NADP}:2 \wedge (\neg\gene{mTRX} \vee \gene{mTRX}:1)) \vee \phantom{0}&
\\*
&&  \qquad \gene{mNADPH\_NADP}:1&
\\*
& 2 & \gene{mTRX}:2 \wedge \gene{mNADPH\_NADP}:2&
\\\cline{1-3}
\multirow{7}{*}{\gene{cCa}}
& 1 & (\gene{IP3R} \wedge \neg \gene{ORAI1} \wedge \neg \gene{TRPM2} \wedge \neg \gene{PMCA})  \vee \phantom{0}&
\\*
&&  \qquad (\neg \gene{IP3R} \wedge \neg \gene{ORAI1} \wedge \gene{TRPM2} \wedge \neg \gene{PMCA}) \vee \phantom{0}&
\\*
&&  \qquad (\neg \gene{IP3R} \wedge \gene{ORAI1} \wedge \gene{TRPM2} \wedge \neg \gene{PMCA}) \vee \phantom{0}&
\\*
&&  \qquad (\neg \gene{IP3R} \wedge \gene{ORAI1} \wedge \neg \gene{TRPM2} \wedge \neg \gene{PMCA})&
\\*
& 2 & (\gene{IP3R} \wedge \gene{ORAI1} \wedge \gene{TRPM2} \wedge \neg \gene{PMCA})\vee \phantom{0}&
\\*
&&  \qquad (\gene{IP3R} \wedge \gene{ORAI1} \wedge \neg \gene{TRPM2} \wedge \neg \gene{PMCA}) \vee \phantom{0}&
\\*
&&  \qquad (\gene{IP3R} \wedge \neg \gene{ORAI1} \wedge \gene{TRPM2} \wedge \neg \gene{PMCA})&
\\\cline{1-3}
\multirow{4}{*}{\gene{cGR}}
& 1 & (\gene{cNADPH\_NADP}:1 \vee \gene{cNADPH\_NADP}:2) \wedge \phantom{0}& 
\\*
&& \qquad (\neg\gene{cGSH\_GSSG} \vee \gene{cGSH\_GSSG}:1)&
\\*
& 2 & (\gene{cNADPH\_NADP}:1 \vee \gene{cNADPH\_NADP}:2) \wedge \phantom{0}&
\\*
&&  \qquad \gene{cGSH\_GSSG}:2&
\\\cline{1-3}
\multirow{3}{*}{\gene{cTR}}
& 1 & (\gene{cNADPH\_NADP}:2 \wedge (\neg\gene{cTRX} \vee \gene{cTRX}:1)) \vee \phantom{0}&
\\*
&&  \qquad \gene{cNADPH\_NADP}:1&
\\*
& 2 & \gene{cTRX}:2 \wedge \gene{cNADPH\_NADP}:2&
\\\hline
\multirow{2}{*}{\gene{Ras}}
& 1 & (\gene{Sos}:1 \wedge \neg (\gene{Sty} \wedge \gene{Gap1})) \vee 
(\gene{Gap1} \wedge \gene{Sty} \wedge \gene{Sos}:2)& \multirow{12}{*}{\cite{Mbo13}}
\\*
& 2 & \gene{Sos}:2 \wedge \neg (\gene{Sty} \wedge \gene{Gap1}) &
\\\cline{1-3}
\multirow{2}{*}{\gene{Tkv}}
& 1 & (\gene{Dpp}:1 \vee \gene{Scw} \vee \gene{Gbb} ) \wedge \gene{Punt} \wedge \neg (\gene{Sog} \vee \gene{Tsg})&
\\*
& 2 & \gene{Dpp}:2 \wedge \gene{Punt} \wedge \neg (\gene{Sog} \vee \gene{Tsg})&
\\\cline{1-3}
\multirow{3}{*}{\gene{Hop}}
& 1 & \gene{Dome} \wedge \neg \gene{ET}
\wedge [(\gene{Stam} \wedge \gene{Hrs} \wedge \gene{Socs36E}) \vee \phantom{0}& 
\\*
&& \qquad (\neg (\gene{Stam} \wedge \gene{Hrs}) \wedge \neg \gene{Socs36E})] &
\\*
& 2 & \gene{Dome} \wedge \neg \gene{ET} \wedge \gene{Stam} \wedge \gene{Hrs} \wedge \neg \gene{Socs36E}&
\\\cline{1-3}
\multirow{3}{*}{\gene{MadMed}}
& 1 & ((\gene{Tkv}:1 \vee \gene{Sax}:1) \wedge \neg \gene{Dad}:1 \wedge \neg \gene{Tkv}:2) \vee \phantom{0}&
\\*
&& \qquad (\gene{Tkv}:2 \wedge\gene{Dad}:1)&
\\*
& 2 & \gene{Tkv}:2 \wedge\neg  \gene{Dad}:1&
\\\cline{1-3}
\multirow{2}{*}{\gene{E\_Spl}}
& 1 & \neg (\gene{Nicd} \wedge \gene{Mam})&
\\*
& 2 &  \gene{Nicd} \wedge \gene{Mam} \wedge \neg \gene{H} \wedge \neg \gene{Gro} \wedge \neg \gene{CtBP} &

\\\hline
\multirow{2}{*}{\gene{VIM}}
& 1 & (\gene{SNAIL1} \vee \gene{ZEB1}) \wedge \neg (\gene{ZEB1} \wedge \gene{SNAIL1})& \multirow{6}{*}{\cite{Sil20}}
\\*
& 2 & \gene{ZEB1} \wedge \gene{SNAIL1}&
\\\cline{1-3}
\multirow{2}{*}{\gene{CDH1}}
& 1 & \neg \gene{ZEB1} \wedge \neg \gene{SNAIL1}&
\\*
& 2 & (\gene{SNAIL1} \vee \gene{ZEB1}) \wedge \neg (\gene{ZEB1} \wedge \gene{SNAIL1})&
\\\cline{1-3}
\multirow{2}{*}{\gene{EMT}}
& 1 & (\gene{VIM} \vee \neg \gene{CDH1}) \wedge \neg (\gene{VIM} \wedge \neg \gene{CDH1})&
\\*
& 2 & \gene{VIM} \wedge \neg \gene{CDH1}&
\\\hline\multirow{3}{*}{\gene{STAT5}}
& 1 & \neg (\gene{IL4R}:2 \vee \gene{IL2R}:2) \wedge \phantom{0}& \multirow{3}{*}{\cite{Nal10}}
\\*
&& \qquad (\gene{IL4R}:1 \vee \gene{IL2R}:1 \vee \gene{IL15R})&
\\*
& 2 & \gene{IL4R}:2 \vee \gene{IL2R}:2&
\\\hline
\caption{\CORRECTION{\label{tab:MVrulesC1} Logical rules associated to multivalued genes from published gene regulatory network models considered in \cite{Sub22}, which are in situation \ref{situA}, \ref{situB} or \ref{situC} and NOT satisfy structure (\ref{C1}).}}
\end{longtable}

\newpage

\begin{longtable}{|L|L|L|L|L|}
\hline
\text{Gene } i & x_i & \text{Rule} & \text{Ref}
\\\hline \endfirsthead
\hline
\text{Gene } i & x_i & \text{Rule \hspace{3cm} (continued)} & \text{Ref}
\\\hline \endhead
\multirow{5}{*}{\gene{mROS}}
& 1 & ((\neg \gene{mPRX} \wedge \gene{mGPX}) \vee 
(\gene{mPRX} \wedge \neg \gene{mGPX})) \wedge \phantom{0} & \multirow{18}{*}{\cite{San19}}
\\* 
&& \qquad (\gene{ETC} \vee \gene{mGR} \vee \gene{mTR}) \wedge \neg \gene{cROS}&
\\*
& 2 & [(\gene{ETC} \vee \gene{mGR} \vee \gene{mTR} \vee \gene{cROS}) \wedge \phantom{0}&
\\* 
&& \qquad \neg (\gene{mPRX} \vee \gene{mGPX})] \vee [\gene{cROS} \wedge \phantom{0}&
\\* 
&& \qquad (\gene{mPRX} \vee \gene{mGPX} \vee \gene{ETC} \vee \gene{mGR} \vee \gene{mTR})]&
\\\cline{1-3}
\multirow{6}{*}{\gene{cROS}}
& 1 & ((\gene{cPRX} \wedge \neg \gene{cGPX}) \vee (\neg \gene{cPRX} \wedge \gene{cGPX})) \wedge \phantom{0}&
\\*
&&  \qquad (\gene{NOX2} \vee \gene{DUOX1} \vee \gene{cTR} \vee \gene{cGR}) \wedge \neg \gene{mROS}&
\\*
& 2 & [\neg (\gene{cPRX} \vee \gene{cGPX}) \wedge \phantom{0}&
\\*
&& \qquad (\gene{NOX2} \vee \gene{mROS} \vee \gene{DUOX1} \vee \gene{cTR} \vee \gene{cGR})] \vee \phantom{0}&
\\*
&& \qquad [\gene{mROS} \wedge (\gene{cPRX} \vee \gene{cGPX} \vee \phantom{0}&
\\*
&& \qquad \gene{NOX2} \vee \gene{DUOX1} \vee \gene{cTR} \vee \gene{cGR})]&
\\\cline{1-3}
\multirow[b]{3}{*}{\gene{cNADPH}}
& 1 & (\gene{mShuttle} \wedge \neg \gene{PPP} \wedge \neg \gene{GLUTAMINOLYSIS})   \vee \phantom{0} & 
\\*
&& \qquad (\neg \gene{mShuttle} \wedge \gene{PPP} \wedge \neg \gene{GLUTAMINOLYSIS}) \vee \phantom{0}&
\\*
&& \qquad (\neg \gene{mShuttle} \wedge \neg \gene{PPP} \wedge \gene{GLUTAMINOLYSIS})&
\\*
\multirow[t]{3}{*}{\quad \gene{\_NADP}}
& 2 & (\gene{mShuttle} \wedge \gene{PPP} \wedge \gene{GLUTAMINOLYSIS})\vee \phantom{0}&
\\*
&& \qquad (\neg \gene{mShuttle} \wedge \gene{PPP} \wedge \gene{GLUTAMINOLYSIS}) \vee \phantom{0}&
\\*
&& \qquad (\gene{mShuttle} \wedge \neg \gene{PPP} \wedge \gene{GLUTAMINOLYSIS}) \vee \phantom{0}&
\\*
&&  \qquad (\gene{mShuttle} \wedge \gene{PPP} \wedge \neg \gene{GLUTAMINOLYSIS})&
\\\hline
\multirow{4}{*}{\gene{Der}}
& 1 & [(\gene{Spi}:1 \vee \gene{Vein}) \wedge \neg \gene{Aos}:1 \wedge \neg \gene{Kek}:1 \wedge \phantom{0}& \multirow{4}{*}{\cite{Mbo13}}
\\*
&&\qquad \neg \gene{Cbl} \wedge \neg \gene{Spi}:2 \wedge \gene{Shc}] \vee \phantom{0}&
\\*
&& \qquad [\gene{Spi}:2 \wedge \gene{Shc} \wedge (\gene{Kek}:1 \vee \gene{Aos}:1 \vee \gene{Cbl})] &
\\*
& 2 & \gene{Spi}:2  \wedge  \neg \gene{Kek}:1  \wedge  \neg \gene{Aos}:1  \wedge  \neg \gene{Cbl}  \wedge  \gene{Shc}&
\\\hline
\multirow{5}{*}{\gene{E2F1}}
& 1 & [(\neg (\gene{CHEK}:2 \wedge \gene{ATM}:2) \wedge (\gene{RAS} \vee \gene{E2F3})) \vee \phantom{0}& \multirow{5}{*}{\cite{Rem15}}
\\*
&& \qquad (\gene{CHEK}:2 \wedge \gene{ATM}:2 \wedge \neg \gene{RAS} \wedge \gene{E2F3}:1)] \wedge \phantom{0}&
\\*
&& \qquad \neg \gene{RB1} \wedge \neg \gene{RBL2}&
\\*
& 2 & \neg \gene{RB1} \wedge \neg \gene{RBL2}\wedge \gene{ATM}:2 \wedge \gene{CHEK}:2 \wedge \phantom{0}&
\\*
&& \qquad (\gene{RAS} \vee \gene{E2F3}:2)&
\\\hline
\multirow{6}{*}{\gene{Spi1}}
& 1 & [\gene{Spi1} \wedge \gene{Runx1} \wedge \phantom{0}&\multirow{6}{*}{\cite{Col17}}
\\*
&& \qquad \neg (((\gene{Cebpa} \vee \gene{Cebpb}) \wedge \gene{Csf1r}) \vee \phantom{0}& 
\\*
&& \qquad (\gene{Foxo1} \wedge \neg \gene{Ikzf1}) \vee \neg \gene{Gfi1})]&
\\*
&& \quad\; \phantom{0} \vee (\gene{Foxo1} \wedge \gene{Ebf1} \wedge \gene{Ikzf1} \wedge \neg (\gene{Spi1}\vee \gene{Runx1}))&
\\*
& 2 & (\gene{Spi1} \wedge \gene{Runx1} \wedge (\gene{Cebpa} \vee \gene{Cebpb} ) \wedge \gene{Csf1r}) \vee \phantom{0}&
\\*
&& \qquad (\gene{Spi1} \wedge \gene{Runx1}\wedge \neg (\gene{Gfi1} \vee (\gene{Foxo1} \wedge   \gene{Ikzf1})))&
\\\hline
\caption{\CORRECTION{\label{tab:MVrulesC2} Logical rules associated to multivalued genes from published gene regulatory network models considered in \cite{Sub22}, which are in situation \ref{situD} and NOT satisfy structure (\ref{C1}).}}
\end{longtable}

\normalsize

\newpage

\CORRECTION{
\section{Method for computing an upper bound on the proportion of SNC functions}
\label{comment}

We first observe that, by Proposition \ref{prop:charac}, each SNC function $f:\prod_{i=1}^n\Omega_i\rightarrow \Z/3\Z$ can be defined by a sequence of at most $K=\sum_i(k_i-1)$ steps of canalization ($k_i=|\Omega_i|$). This representation is obviously not unique, and can be improved by grouping some commuting steps: for instance, consecutive steps which fix $f(x)$ to the same value commute. This leads to a (still \CORRECT{non-unique}) representation of $f$ by a matrix
$$
(m_i^j,M_i^j,v^j)_{i\in\{1,\ldots,n\}}^j
$$
where the step number $j$ is at most $K$: at each step $j$, canalization of variables $x_i$ ($i=1,\ldots,n$)
\begin{itemize}
\item $m_i^j$ times to its minimal value in its range
\item and $M_i^j$ times to its maximal value
\end{itemize}
leads to $f(x)=v^j$. Consecutive values $v^j,v^{j+1}$ are assumed to be different, except possibly for the last two values. For each $i$, $\sum_j m_i^j+M_i^j=k_i-1$, so the above matrices can be produced by enumerating lists of given length $K$ whose sum equals a given number, \ie ordered integer partitions (with fixed length and possibly null entries). Then void steps (in which $m_i^j=M_i^j=0$ for all $i$) are removed.
}

\CORRECTION{
\section{Python program computing the upper bound on the proportion of SNC functions}
\label{python}
}

\lstset{language=Python}
\begin{lstlisting}
import numpy as np
import itertools
import math
def sums(length, totalsum):
    # all lists of given length whose sum equals totalsum
    if length == 1:
        yield (totalsum,)
    else:
        for v in range(totalsum + 1):
            for p in sums(length - 1, totalsum - v):
                yield (v,) + p
def removezeroes(lst,n):
    # lst = list of length n "rows"
    # removes zero rows in lst
    return [x for x in lst if x != [0] * n]
def upSNC(arities):
    # arities = list of arities k_i of variables x_1,...,x_n
    n = len(arities)
    # n : nb of variables
    totalsums = [arities[i]-1 for i in range(n)]
    # totalsums = list of max x_i (= k_i - 1)
    length = sum(totalsums)
    # length = K = sum of the k_i - 1
    l = [list(sums(length, totalsums[i])) for i in range(n)]
    # l lists all SNC decompositions (lists of steps)
    # in columns
    s = set()
    for pr in itertools.product(*l):
        step = removezeroes(np.array(pr).T.tolist(),n)
        s.add(tuple(tuple(i) for i in step))
    # s = set of SNC decompositions in rows
    # after removing zero rows
    ones = [2**sum([e.count(1) for e in x]) for x in s]
    # nb of possibilities for binary inputs
    twos = [3**sum([e.count(2) for e in x]) for x in s]
    # nb of possibilities for ternary inputs
    values = [9 * 2**(len(x)-1) for x in s]
    # 3 * 2^(nb steps - 1) * 3
    return sum([a*b*c for a,b,c in zip(ones,twos,values)])
def uppropSNCbyarity(arities):
    p = 1
    for a in arities:
        p *= a
    return upSNC(arities)/3**p
def uppropSNCbynbvars(n):
    snc = 0
    for i in range(n+1):
        print('upSNC('+str([2] * (n-i) + [3] * i)+')...')
        snc += upSNC([2] * (n-i) + [3] * i) * math.comb(n,i)
    tot = 0
    for i in range(n+1):
        tot += 3**(2**i * 3**(n-i)) * math.comb(n,i)
    return snc/tot
\end{lstlisting}
\end{document}